\documentclass[12pt,a4paper,pdfps]{article}
\usepackage[cp1251]{inputenc}
\usepackage[english]{babel}
\usepackage{amsmath,amssymb,amsthm}
\usepackage{graphicx}
\usepackage{comment}
\usepackage{enumitem}
\usepackage{color}
\usepackage[colorlinks,linkcolor=blue,filecolor=blue,citecolor=blue]{hyperref}
\DeclareMathOperator{\ad}{\mathrm{ad}}
\DeclareMathOperator{\Ad}{\mathrm{Ad}}
\DeclareMathOperator{\Exp}{\mathrm{Exp}}
\DeclareMathOperator{\sspan}{\mathrm{span}}
\DeclareMathOperator{\codim}{\mathrm{codim}}
\DeclareMathOperator{\Aut}{\mathrm{Aut}}
\newcommand{\tmax}{t_{\mathrm{max}}}
\newcommand{\tconj}{t_{\mathrm{conj}}}
\newcommand{\g}{\mathfrak{g}}
\newcommand{\R}{\mathbb{R}}

\newcommand{\id}{\mathrm{id}}
\newcommand{\SU}{\mathrm{SU}}
\newcommand{\SL}{\mathrm{SL}}
\newcommand{\PSL}{\mathrm{PSL}}
\newcommand{\SO}{\mathrm{SO}}

\newcommand{\SE}{\mathrm{SE}}
\newcommand{\SH}{\mathrm{SH}}
\newcommand{\Hl}{\vec{H}}
\newcommand{\Hh}{\Hl_{\mathrm{hor} }}

\newcommand{\Hv}{\Hl_{\mathrm{vert}}}
\newcommand{\Hr}{\Hl_R}
\newcommand{\RH}{H_R}

\theoremstyle{definition}
\newtheorem{definition}{Definition}
\newtheorem{remark}{Remark}
\newtheorem{example}{Example}

\theoremstyle{plain}
\newtheorem{corollary}{Corollary}
\newtheorem{lemma}{Lemma}
\newtheorem{theorem}{Theorem}
\newtheorem{proposition}{Proposition}

\title{Symmetries in left-invariant optimal control problems
\footnote{This work is supported by the Russian Science Foundation
under grant 17-11-01387 and performed in A.\,K.~Ailamazyan Program Systems
Institute of Russian Academy of Sciences.}
}

\author{
A.\,V.~Podobryaev \\ A.\,K.~Ailamazyan Program Systems
Institute of RAS \\ \tt{alex@alex.botik.ru}
}

\begin{document}

\maketitle

\begin{abstract}
We consider left-invariant optimal control problems on connected Lie groups.
We describe the symmetries of the exponential map that are induced by the symmetries of the vertical part of the Hamiltonian system of Pontryagin maximum principle.
These symmetries play a key role in investigation of optimality of extremal trajectories.
For connected Lie groups such that generic coadjoint orbit has codimension not more than 1 and connected stabilizer
we introduce a general construction for such symmetries of the exponential map.



\end{abstract}

\section*{\label{section-introduction}Intoduction}
Geometric control theory (see for example~\cite{agrachev-sachkov}) deals with left-invariant optimal control problems on a Lie group $G$. Consider a family of left-invariant vector fields $F_u$ that depend analytically on $u \in U \subset \R^n$. Consider also  a left-invariant analytic function $\varphi: G\times U \rightarrow \R$, a point $q_1 \in G$, and a fixed time $t_1 > 0$. The problem is to find a control $u \in  L^{\infty}([0, t_1], U)$ and a Lipschitz curve $q_u : [0, t_1] \rightarrow G$ such that
\begin{equation}
\label{eq-optimal-control-problem}
\int_0^{t_1} \varphi(q_u(t), u(t))\, dt \rightarrow \min, \quad \dot{q}_u(t) = F_{u(t)}(q_u(t)), \quad q_u(0) = \id,
\quad q_u(t_1) = q_1 \in G.
\end{equation}
Consider functions $h_u$ on the cotangent bundle $T^*G$ that depend on parameter $u\in U$:
$$
h_u(\lambda) = \lambda(F_u(\pi(\lambda))) - \varphi(\pi(\lambda), u), \qquad \lambda\in T^*G,
$$
where $\pi: T^*G \rightarrow G$ is the natural projection. Assume that for all $\lambda\in T^*G$ the quadratic form $\frac{\partial^2}{\partial u^2}h_u(\lambda)$ is negative definite and the function $u \mapsto h_u(\lambda)$ has maximum. Then via Pontryagin maximum principle~\cite{pontryagin, agrachev-sachkov} we obtain a Hamiltonian differential equation on the cotangent bundle $T^*G$, such that its phase curves project to optimal trajectories on the group $G$:
\begin{equation}
\label{eq-Hamiltonian-system}
\dot{\lambda} = \Hl(\lambda), \qquad \pi(\lambda(t)) = q_u(t), \qquad \lambda: [0, t_1] \rightarrow T^*G,
\end{equation}
where $H(\lambda) = \max_{u\in U}{h_u(\lambda)}$ is the analytic maximized Hamiltonian of Pontryagin maximum principle, $\Hl$ is the corresponding analytic Hamiltonian vector field. The curve $\lambda(t)$ is called \emph{a normal extremal}. Next we will consider only such extremals. The curve $q_u(t)$ is called \emph{a normal extremal trajectory}.

If we have an explicit solution of differential equation~\eqref{eq-Hamiltonian-system}, then we have a parametri\-zation of extremal trajectories. After that it remains to study optimality of extremal trajectories.

\begin{definition}
\label{def-Maxwell-point-and-time}
\emph{A Maxwell point for an optimal control problem~$~\eqref{eq-optimal-control-problem}$} is a point where two distinct extremal trajectories meet one another with the same value of the cost functional and the time. This time is called \emph{a Maxwell time}.
\end{definition}

It is well known (see for example~\cite{sachkov-didona1}), that an extremal trajectory can not be optimal after a Maxwell point. That is why description of Maxwell points plays an important role in investigation of optimality of extremal trajectories. In particular, the first Maxwell time is an upper bound for the time of loss of optimality (\emph{the cut time}).

A natural reason of appearance of Maxwell points is a symmetry of extremal trajectories.
Let us give definitions.

\begin{definition}
\label{def-exponential-map}
\emph{The exponential map of problem~$\eqref{eq-optimal-control-problem}$} is the map
$$
\Exp : \g^* \times \R_+ \rightarrow G, \qquad \Exp{(p, t)} = \pi \circ e^{t \Hl} (\id, p), \qquad (p, t) \in \g^* \times \R_+,
$$
where $\g$ is the Lie algebra of the Lie group $G$, and $e^{t \Hl}$ is the flow of the Hamiltonian vector field $\Hl$.
\end{definition}

\begin{definition}
\label{def-symmetry-of-exponential-map}
\emph{A symmetry of the exponential map} is a pair of diffeomorphisms
$$
s: \mathcal{W} \times \R_+ \rightarrow \mathcal{W} \times \R_+, \quad S: G \rightarrow G \quad \text{such that} \quad \Exp \circ s = S \circ \Exp,
$$
where $\mathcal{W} \subset \g^*$ is an open dense subset.
\end{definition}

Consider the trivialization of the cotangent bundle via left shifts:
$$
\tau: G\times\g^*\rightarrow T^*G, \qquad \lambda = \tau(g, p) = dL_{g^{-1}}^* (p) \in T^*_gG, \qquad g \in G, \qquad p \in \g^* = T^*_{\id}G.
$$
where $L_g:G\rightarrow G$ is the left shift by the element $g\in G$.

The Hamiltonian $H$ is left-invariant, so we assume that $H\in
C^{\infty}(\g^*)$. A Hamiltonian vector field is a sum of {\it
the horizontal} and {\it the vertical parts}~\cite{agrachev-sachkov}:
\begin{equation}
\label{eq-hor-vert-parts}
\begin{array}{c}
\Hl(\tau(g, p)) = d_{(g, p)}\tau (\Hh(g, p) + \Hv(p)),\\
\Hh(g, p) = dL_g d_p H, \qquad \Hv(p) = (\ad^* d_p H)p,\\
\end{array}
\end{equation}
where $d_p H \in T_p^*\g^* \simeq \g$ is the differential of $H$ at a point $p$.

The Hamiltonian system $\dot{\lambda} = \Hl(\lambda)$ is triangular (its vertical part is independent of state variables). So, one can naturally consider symmetries of the exponential map induced by symmetries of the vertical part of the Hamiltonian system (see complete statement in Theorem~\ref{th-symmetry}).

A plan of investigation of optimality of extremal trajectories reads as follows.
\begin{enumerate}
\item Parametrization of extremal trajectories.
\item Description of symmetries of the vertical part of the Hamiltonian system. Extension of these symmetries to symmetries of the exponential map.
\item Search for Maxwell points that correspond to symmetries. Search for the first Maxwell time as a function $\tmax : \g^* \rightarrow \R_+ \cup \{+\infty\}$.
\item Estimation of the first conjugate time, i.e., the function $\tconj : \g^* \rightarrow \R_+ \cup \{+\infty\}$ such that a pair $(p, \tconj{(p)})$ is a critical point of the exponential map.
\item Verification of the condition $\tmax{(p)} \leqslant \tconj{(p)}$ for almost all $p \in \g^*$.
\item Application of the Hadamard theorem on global diffeomorphism~\cite{krantz-parks} to the map
    $$
    \Exp{(\cdot, t_1)}: \{p\in\g^*\setminus 0 \ | \ t_1 < \tmax(p)\} \rightarrow G \setminus (\{\id\}\cup\overline{\mathcal{M}}),
    $$
    where $\overline{\mathcal{M}}$ is the closure of the Maxwell set.(A smooth non-degenerate proper map of connected and simply connected manifolds of equal dimensions is a diffeomorphism.)
\end{enumerate}
We need items~4 and 5 to verify the non-degenerateness condition of the Hadamard theorem.
If implementation of all these steps is complete, than the first Maxwell time is actually the cut time.

Notice, that implementation of this program is not guaranteed.
For example, symmetries of the vertical part of the Hamiltonian system may not produce a complete description of the Maxwell set. Such situation appears in Euler elasticae problem~\cite{sachkov-elasticae, ardentov-elasticae}.
However this method works in several sub-Riemannian~\cite{agrachev-barilari-boscain} and Riemannian problems (see references below):

\begin{enumerate}
\item
\label{pr-Heisenberg}
Sub-Riemannian problem on Heisenberg type group \\ (C.~Autenried, M.~Godoy~Molina~\cite{autenried-molina}).
\item
\label{pr-nilpotent-3-6}
Free nilpotent sub-Riemannian problem with growth vector $(3, 6)$ \\ (O.~Myasnichenko~\cite{myasnichenko}
and independently A.~Montanari and D.~Morbidelli~\cite{montanari-morbidelli}, also some results in general case of free step two
Carnot group achieved by L.~Rizzi and U.~Serres~\cite{rizzi-serres}).
\item
\label{pr-Dido}
Generalized Dido problem (Yu.\,L.~Sachkov~\cite{sachkov-didona, sachkov-didona1, sachkov-didona2}).
\item
\label{pr-Engel}
Sub-Riemannian problem on Engel group (A.\,A.~Ardentov, Yu.\,L.~Sachkov~\cite{ardentov-sachkov-engel-1, ardentov-sachkov-engel-2, ardentov-sachkov-engel-3, ardentov-sachkov-engel-4}).
\item Sub-Riemanian problems on the Lie groups $\SL_2(\R)$, $\PSL_2(\R)$, $\SO_3$, $\SU_2$
      (U.~Boscain, F.~Rossi~\cite{boscain-rossi}, and independently, using another techniques by V.\,N.~Berestovskii and I.\,A.~Zubareva \cite{berestovskij-zubareva-sl2, berestovskij-zubareva-so3}),
      also C.~Autenried and I.~Markina considered some generalizations to Shtiefel manifolds~\cite{autenried-markina}.
\item Riemannian problems on the Lie groups $\SL_2(\R)$, $\PSL_2(\R)$, $\SO_3$, $\SU_2$
      (A.\,V.~Podob\-rya\-ev, Yu.\,L.~Sachkov~\cite{podobryaev-sachkov-1, podobryaev-sachkov-2}).
\item
\label{pr-se2}
Sub-Riemannian problem on the Lie group $\SE_2$ (Yu.\,L.~Sachkov~\cite{sachkov-moiseev-se-1, sachkov-se-2, sachkov-se-3}, the first paper in collaboration with I.~Moiseev).
\item
\label{pr-sh2}
Sub-Riemannian problem on the Lie group $\SH_2$ (Ya.\,A.~Butt, Yu.\,L.~Sachkov, \\ A.\,I.~Bhatti~\cite{sh2-1, sh2-2}).
\item
\label{pr-sphere-plane}
The problem of a rolling sphere on the plane without twisting and slipping \\ (Yu.\,L.~Sachkov~\cite{sachkov-sphere-plane}).

\end{enumerate}

Here we have the problems on nilpotent groups (\ref{pr-Heisenberg}--\ref{pr-Engel}), compact groups ($\SO_3$, $\SU_2$), semisimple groups ($\SL_2(\R)$, $\PSL_2(\R)$),
semidirect product of commutative and compact groups (\ref{pr-se2}, $\SE_2 = \R^2 \leftthreetimes \SO_2$),
semidirect product of commutative and nilpotent groups (\ref{pr-sh2}, $\SH_2 = \R^2 \leftthreetimes \R$),
direct product of compact and commutative groups (\ref{pr-sphere-plane}, $\SO_3 \times \R^2$).

Left-invariant optimal control problems on nilpotent Lie groups are of special interest due to  existence of a nilpotent approximation~\cite{agrachev-sarychev} of control systems.

In the problems listed above an extension of symmetries of the vertical subsystem to symmetries of the exponential map was constructed by explicit formulas for the map $\Exp$ (i.e., an explicit parametrization of extremal trajectories) or by an explicit form of the Hamiltonian system. Existence of such extension was not guaranteed a priori.

In Section~\ref{section-results} we introduce conditions for existence of extension of symmetries of the vertical subsystem to  symmetries of the exponential map. Also there is a general construction of such symmetries and some corollaries. The proof is in Section~\ref{section-proof}. We describe a non-trivial example in Section~\ref{section-example}.

\section{\label{section-results}The main result}

Let $G$ be a connected Lie group, $\g$ be its Lie algebra. Consider the cotangent bundle $T^*G$ with the action of the group $G$ by left shifts. Let $H \in C^{\infty}(T^*G)$ be a left-invariant Hamiltonian, $\Hl$ be the corresponding Hamiltonian vector field, $\Hh$ and $\Hv$ be its horizontal and vertical parts, respectively, see~(\ref{eq-hor-vert-parts}).

\begin{theorem}
\label{th-symmetry}
Let $G$ be a connected Lie group, such that generic stabilizer of the coadjoint action is connected
and has dimension not more than 1.
Assume that $H : T^*G \rightarrow \R$ is a left-invariant Hamiltonian, and an operator $\sigma^*:\g^*\rightarrow\g^*$ is such that
$\sigma^*$
preserves the Hamiltonian $H$ and there holds one of the two conditions:\\
{\rm (a)} $\sigma^*(\Hv) = \Hv$ and $\sigma$ is an automorphism of the Lie algebra $\g$;\\
{\rm (b)} $\sigma^*(\Hv) = -\Hv$ and $\sigma$ is an anti-automorphism of the Lie algebra $\g$.\\
Then the pair of diffeomorphisms $(s, S^{-1})$ is a symmetry of the exponential map, where
$$
s(p, t) = \left\{
\begin{array}{lc}
(\sigma^* p, t), & \text{in case {\rm (a)},} \\
(\sigma^* e^{t\Hv} p, t), & \text{in case {\rm (b)},} \\
\end{array}
\right.
$$
and $S: G \rightarrow G$ is the (anti-)automorphism of the Lie group such that $d_{\id}S = \sigma$.
\end{theorem}

\begin{remark}
\label{rem-coadjoint-orbit-condition}
The condition for generic stabilizer of the coadjoint action plays role only for symmetries of case~(b).
\end{remark}

\begin{remark}
\label{rem-case-b-is-main}
In case~(b) if $\sigma$ is an anti-automorphism, then $-\sigma$ is an automorphism and one can construct a symmetry as in case~(a).
But it is not what we want. We construct not any symmetry, but a special kind of symmetry for case~(b).
Notice that in case~(b) the curve
$$
\{\Exp \circ s (p, \tau) = \pi \circ e^{\tau\Hl} \circ \sigma^* \circ e^{\tau \Hv} (\id, p) \ | \ \tau \in \R_+ \}
$$
is not an extremal trajectory, unlike the situation in case~(a).
The segment of extremal trajectory $\{\Exp (\sigma^* e^{t \Hv} p, \tau) \ | \ \tau \in [0, t]\}$ is symmetric to the segment of initial extremal trajectory $\{\Exp (p, \tau) \ | \ \tau \in [0, t]\}$.

Usually symmetries of case~(b) are essential for construction of Maxwell strata in applications.
In examined situations there is an open subset of extremal trajectories that intersect the corresponding Maxwell stratum.
\end{remark}

Everywhere below we consider symmetries $(s, S^{-1})$ of the exponential map such that $\sigma$ satisfies the hypotheses of Theorem~\ref{th-symmetry}. We call $\sigma^*$ \emph{the symmetry of the vertical part of the Hamiltonian vector field.}

\begin{definition}
\label{def-maxwell-sets}
\emph{The Maxwell sets in the pre-image and in the image of exponential map, corresponding to the symmetry $(s, S^{-1})$,} are the sets
$$
M_{\sigma^*} = \{ (p, t_1) \in \mathcal{W}\times\R_+ \ | \ \Exp(p, t_1) = \Exp \circ s (p, t_1) \}, \qquad \mathcal{M}_{\sigma^*} = \Exp{M_{\sigma^*}},
$$
respectively.
\end{definition}

\begin{remark}
\label{rem-uni-Maxwell-set}
It follows directly from Theorem~\ref{th-symmetry} that
the Maxwell set in the image of the exponential map $\mathcal{M}_{\sigma^*}$ is a subset of the set of fixed points $G^{S}$
of the mapping $S$.
Let $F_{\sigma^*}(g) = 0$ be an equation of $G^{S}$, where $F_{\sigma^*} : G\rightarrow \R^k$, for some $k$.
This allows us to determine the Maxwell time
$t_{\mathrm{max}}^{\sigma^*}:\g^*\rightarrow\R_+\cup\{+\infty\}$, corresponding to the symmetry $\sigma^*$,
by the implicit function
$$
F_{\sigma^*}(\Exp{(p, t_{\mathrm{max}}^{\sigma^*}{(p)})}) = 0.
$$
To find the set of the first Maxwell points corresponding to symmetries we need to investigate these implicit functions for different symmetries to find the first Maxwell time.
\end{remark}

\begin{corollary}
\label{crl-two-leftinv-prob}
Assume that there are two left-invariant optimal control problems on a Lie group and a symmetry of the vertical parts of both Hamiltonian vector fields. If two extremal trajectories corresponding to these problems meet one another, then the symmetric trajectories meet one another as well with the same values of time and cost functional.
\end{corollary}

\begin{corollary}
\label{crl-Killing-geodesics}
Let $G$ be a connected compact Lie group. If a normal extremal trajectory meets a geodesic of the Killing metric, then the symmetric extremal trajectory meets the symmetric geodesic of the Killing metric at the same instant of time.
\end{corollary}

\begin{proof}
Corollary~\ref{crl-two-leftinv-prob} immediately follows from Theorem~\ref{th-symmetry}.

The vertical part of the Hamiltonian vector field is trivial for the Riemannian problem for the Killing metric (see, for example,~\cite{agrachev-sachkov}). Thus, any symmetry of the vertical part of the Hamiltonian system for the optimal control problem is a symmetry of the vertical part of the Hamiltonian vector field for the Riemannian problem for the Killing metric.
Corollary~\ref{crl-Killing-geodesics} follows from Corollary~\ref{crl-two-leftinv-prob}.
\end{proof}

\begin{remark}
For any element $g \in G$ of a compact connected Lie group there exists $\xi \in \g$ such that $g = \exp{(\xi)}$.
By definition we have $S^{-1}(g) = \exp{(\sigma^{-1}\xi)}$. This does not depend on a choice of the element $\xi$, because the (anti-)automorphism $\sigma$ preserves fibers of the map $\exp$.
For a compact Lie group $G$ one-parametric subgroups $g_{\xi}(t) = \exp{(t \xi)}$
are geodesics of the bi-invariant Riemannian metric that is defined by the Killing form~\cite{agrachev-sachkov}.
\end{remark}

\section{\label{section-proof}Proof of Theorem~\ref{th-symmetry}}

Define the diffeomorphism $S^* : T^*G \rightarrow T^*G$ as follows
$$
S^*(\lambda) = (d_{S^{-1}(g)}S)^*(\lambda), \qquad \lambda \in T^*G, \qquad \pi(\lambda) = g \in G.
$$
By definition we have
\begin{align}
\pi \circ S^* = S^{-1} \circ \pi, \label{eq-projection} \\
S^*|_{T^*_{\id}G} = \sigma^*.      \label{eq-coalgebra-restriction}
\end{align}
Moreover $S^*$ is symplectomorphism of the canonical symplectic structure $\omega$ on $T^*G$.

\begin{lemma}
\label{lemma-new-flow}
The equality $(S^*)^{-1} \circ e^{t\Hl} \circ S^* = e^{t\overrightarrow{(H \circ S^*)}}$ is satisfied.
\end{lemma}

\begin{proof}
Obviously $(S^*)^{-1} \circ e^{t\Hl} \circ S^* = e^{t(S^*)^{-1}\Hl}$.
Then, by definition of the Hamiltonian vector field $\Hl$, we have
$$
d(H \circ S^*) = dH \circ dS^* = \omega(dS^*(\cdot), \Hl).
$$
Since $S^*$ is a symplectomorphism, this is equal to
$\omega(\cdot, (dS^*)^{-1}\Hl)$. So, by definition of the Hamiltonian vector field $\overrightarrow{(H \circ S^*)}$ we get the statement of Lemma.
\end{proof}

Next we find out how the diffeomorphism $S^*$ acts on the Hamiltonian $H$.

\begin{lemma}
\label{lemma-case-a-keeps-Hamiltonian}
If $\sigma$ is an automorphism of the Lie algebra $\g$ and $\sigma^*$ preserves the left-invariant Hamiltonian $H$,
then $H \circ S^* = H$.
\end{lemma}

\begin{proof}
Since $S$ in an automorphism of the Lie group $G$ we have
\begin{equation}
\label{eq-automorphism-S}
S L_g = L_{S(g)} S \quad \Rightarrow \quad L_g^* S^* = S^* L_{S(g)}^* \quad \Rightarrow \quad L_{S^{-1}(g^{-1})}^* S^* = S^* L_{g^{-1}}^*.
\end{equation}

Consider the action of diffeomorphism $S^*$ on the Hamiltonian $H$:
$$
H(S^* \lambda) = H(S^* L_{g^{-1}}^* p) \stackrel{\eqref{eq-automorphism-S}}{=} H(L_{S^{-1}(g^{-1})}^* S^* p), \qquad
\lambda = L_{g^{-1}}^* p \in T^*_g G, \qquad p \in \g^*.
$$
Due to the left-invariance of the Hamiltonian and formula~\eqref{eq-coalgebra-restriction} this is equal to $H(\sigma^*p)$. But $H$ is invariant under $\sigma^*$, so we get $H(S^*\lambda) = H(p) = H(\lambda)$.
\end{proof}

As a consequence in case~(a) from Lemma~\ref{lemma-new-flow} and Lemma~\ref{lemma-case-a-keeps-Hamiltonian} we obtain
$e^{t\Hl} \circ S^* = S^* \circ e^{t\Hl}$. From equalities~\eqref{eq-projection}--\eqref{eq-coalgebra-restriction} it follows that $\Exp \circ s = S^{-1} \circ \Exp$. So, the pair $(s, S^{-1})$ is a symmetry of exponential map in case~(a).
\medskip

Consider now case~(b).

\begin{lemma}
\label{lemma-case-b-changes-Hamiltonian-to-right-invariant}
If $\sigma$ is an anti-automorphism of the Lie algebra $\g$ and $\sigma^*$ preserves the left-invariant Hamiltonian $H$,
then the function $\RH = H \circ S^*$ is right-invariant.
\end{lemma}

\begin{proof}
Since $S$ is anti-automorphism of the Lie group $G$,
instead of formula~\eqref{eq-automorphism-S} we get
\begin{equation}
\label{eq-anti-automorphism-S}
S L_g = R_{S(g)} S \quad \Rightarrow \quad L_g^* S^* = S^* R_{S(g)}^* \quad \Rightarrow \quad L_{S^{-1}(g^{-1})}^* S^* = S^* R_{g^{-1}}^*,
\end{equation}
where $R_g$ denotes a right-shift by element $g \in G$.

Next, for $\lambda \in T^*_gG$ there exists $p \in \g^*$ such that $\lambda = R_{g^{-1}}^* p$. Then
$$
\RH(\lambda) = H(S^* \lambda) = H(S^* R_{g^{-1}}^* p) \stackrel{\eqref{eq-anti-automorphism-S}}{=} H(L_{S^{-1}(g^{-1})}^* S^* p).
$$
Due to the left-invariance of the Hamiltonian $H$ and formula~\eqref{eq-coalgebra-restriction}
this is equal to $H(\sigma^* p) = H(p)$ (since Hamiltonian $H$ is invariant under $\sigma^*$).
So, $\RH(R_{g^{-1}}^* p) = H(p)$.
\end{proof}

\begin{lemma}
\label{lemma-right-Hamiltonian-vector-field-vertical-part}
The Hamiltonian vector field $\Hr$ is right-invariant and its vertical part in right trivialization of the cotangent bundle is equal to $-\Hv$.
\end{lemma}

\begin{proof}
Since $\RH$ is right-invariant, then $\Hr$ is right-invariant as well.
It follows from the proof of Lemma~\ref{lemma-new-flow} that $\Hr = (S^*)^{-1} \Hl$.
Then its vertical part equals $(\sigma^*)^{-1} \Hv = -\Hv$ (due to the condition of case~(b)).
\end{proof}

Next we use the following notation. For $p \in \g^*$ denote by $G_p = \{ g \in G \ | \ (\Ad^*{g})p = p\}$ the stabilizer
of the covector $p$ with respect to the coadjoint action. Let $\g_p$ be the Lie algebra of the stabilizer $G_p$.

\begin{proposition}
\label{prop-flows-of-left-and-right-Hamiltonian-vector-fields}
Assume that $p_0 \in \g^*$ is such that the stabilizer $G_{p_0}$ is connected, $\dim{G_{p_0}} \leqslant 1$ and
$p_0(\g_{p_0}) \neq 0$.
Then
$$
e^{t\Hr} \circ e^{t\Hv} p_0 = e^{t\Hl} p_0 \quad \text{for all} \quad t \in \R_+.
$$
\end{proposition}

\begin{proof}
Define the curves
\begin{align*}
p & : [0, t] \rightarrow \g^*, \qquad p_{\tau} = e^{\tau\Hv}p_0, \\
p^- & : [0, t] \rightarrow \g^*, \qquad p^-_{\tau} = e^{-\tau\Hv}p_t = p_{t-\tau}.
\end{align*}
Denote by
\begin{align*}
&J_L : T^*G \rightarrow \g^*, \qquad J_L(\lambda) = R_g^*(\lambda), \qquad \lambda \in T^*G, \\
&J_R : T^*G \rightarrow \g^*, \qquad J_R(\lambda) = L_g^*(\lambda), \qquad \lambda \in T^*G,
\end{align*}
the momenta maps of the left and right actions of the Lie group $G$ on $T^*G$, respectively.

From symplectic reduction theory~\cite{marsden-ratiu} it is known that
\begin{enumerate}[label=\arabic*.,ref=\arabic*]
\item \label{it-left-trajectory}
      The trajectory of the left-invariant Hamiltonian vector field $\Hl$ passing through the point $p_0$
      is contained in the level-set $J_L^{-1}(p_0)$.
\item \label{it-right-trajectory}
      The trajectory of the right-invariant Hamiltonian vector field $\Hr$ passing through the point $p_t$
      is contained in the level-set $J_R^{-1}(p_t)$.
\item \label{it-left-bundle}
      The manifold $J_L^{-1}(p_0)$ is a bundle over the coadjoint orbit $(\Ad^*{G}) p_0$.
      The projection map is $J_R : J_L^{-1}(p_0) \rightarrow (\Ad^*{G}) p_0$
      and the fibers are the orbits of the left action of the group $G_{p_0}$.
\item \label{it-right-bundle}
      The manifold $J_R^{-1}(p_t)$ is a bundle over the coadjoint orbit $(\Ad^*{G}) p_t$.
      The projection map is $J_L : J_R^{-1}(p_0) \rightarrow (\Ad^*{G}) p_t$
      and the fibers are the orbits of the right action of the group $G_{p_t}$.
\item \label{it-projection-of-Hamiltonian-vector-fields}
      The projections $J_R$ and $J_L$ take the left- and right-invariant Hamiltonian vector fields to their vertical parts:
      $J_R \Hl = \Hv$, $J_L \Hr = -\Hv$.
\item \label{it-coadjoint-orbits-coincide}
      If $g_L(t) = \pi e^{t\Hl} p_0$, then $(\Ad^* g_L(t)^{-1}) p_0 = p_t$
      (it follows from~\ref{it-left-bundle} and \ref{it-projection-of-Hamiltonian-vector-fields}).
      So, the coadjoint orbits of $p_0$ and $p_t$ coincide one with another.
      Denote this coadjoint orbit by $\mathcal{O}$.
\item \label{it-left-right-intersection}
      The intersection $\Gamma = J_L^{-1}(p_0) \cap J_R^{-1}(p_t)$ is an orbit of the left action of the group $G_{p_0}$
      and at the same time is an orbit of the right action of the group $G_{p_t}$. For any $f \in \pi(\Gamma)$ we have
      $G_{p_t} = f^{-1}G_{p_0}f$.
\end{enumerate}

It follows from~\ref{it-left-trajectory}, \ref{it-right-trajectory}, \ref{it-left-right-intersection} and the assumption of
Proposition~\ref{prop-flows-of-left-and-right-Hamiltonian-vector-fields} that the trajectories of the Hamiltonian vector fields $\Hl$ and $\Hr$ passing through $p_0$ and $p_t$ (respectively) come at time $t$ to the submanifold $\Gamma$ that is connected.
We need to show that these trajectories intersect at time $t$. This is obvious if $\dim{G_{p_0}} = 0$, since submanifold $\Gamma$ is a point. So, assume that $\dim{G_{p_0}} = 1$.

We reconstruct~\cite{marsden-mongomery-ratiu} the trajectories of the Hamiltonian vector fields $\Hl$ and $\Hr$ from the trajectories of the vector fields $\Hv$ and $-\Hv$ at the coadjoint orbit $\mathcal{O}$.

Consider the restrictions of the bundles $J_L^{-1}(p_0)$ and $J_R^{-1}(p_t)$ to the curves
$p([0, t])$ and $p^-([0, t])$ at the coadjoint orbit $\mathcal{O}$.
$$
\rho_L : \mathcal{P}_L = J_L^{-1}(p_0)|_{p([0, t])} \rightarrow p([0, t]), \qquad
\rho_R : \mathcal{P}_R = J_R^{-1}(p_t)|_{p^-([0, t])} \rightarrow p^-([0, t]).
$$

\begin{definition}
Let $E \rightarrow B$ be a principal $G$-bundle. \emph{A principal connection} is a $\g$-valued one-form $A$ on $E$ such that \\
(1) $A(\xi_*) = \xi \in \g$, where $\xi_*$ is a velocity field of $G$-action,\\
(2) $A T_g = (\Ad{g}) A$, where an element $g \in G$ acts by a diffeomorphism $T_g : E \rightarrow E$.\\
A curve $\gamma : [0, t] \rightarrow E$ is called \emph{a horizontal curve} if $A\dot{\gamma}(t) \equiv 0$.
\end{definition}

We make a reconstruction of the curves
\begin{align*}
&\lambda_L : [0, t] \rightarrow \mathcal{P}_L, \qquad \lambda_L(\tau) = e^{\tau\Hl} p_0, \qquad \ \tau \in [0, 1], \\
&\lambda_R : [0, t] \rightarrow \mathcal{P}_R, \qquad \lambda_R(\tau) = e^{\tau\Hr} p_t, \qquad \tau \in [0, 1]
\end{align*}
from the curves $p$, $p^-$ in two steps.

First, we introduce principal connections $A_L$ and $A_R$ on the bundles $\mathcal{P}_L$ and $\mathcal{P}_R$ (respectively)
and find horizontal curves (in sense of these connections)
\begin{align}
&\gamma_L : [0, t] \rightarrow \mathcal{P}_L, \qquad \rho_L \circ \gamma_L = p, \qquad \ \ \gamma_L(0) = p_0, &\label{eq-horizontal-curve-L} \\
&\gamma_R : [0, t] \rightarrow \mathcal{P}_R, \qquad \rho_R \circ \gamma_R = p^-, \qquad \gamma_R(0) = p_t. &\label{eq-horizontal-curve-R}
\end{align}

Second, we look for curves
\begin{align}
&h_L : [0, t] \rightarrow G_{p_0} \ \text{such that} \ \lambda_L(\tau) = L(h_L(\tau)) \gamma_L(\tau) \ \text{for} \ \tau \in [0, t], &\label{eq-addition-L} \\
&h_R : [0, t] \rightarrow G_{p_t} \ \text{such that} \ \lambda_R(\tau) = R(h_R(\tau)) \gamma_R(\tau) \ \text{for} \ \tau \in [0, t], &\label{eq-addition-R}
\end{align}
where by $L$ and $R$ we denote the left- and the right- actions of the groups $G_{p_0}$ and $G_{p_t}$ (respectively) on the bundles
$\mathcal{P}_L$ and $\mathcal{P}_R$.

Now Proposition~\ref{prop-flows-of-left-and-right-Hamiltonian-vector-fields} immediately follows from Lemmas~\ref{lemma-horizontal}, \ref{lemma-addition}.
\end{proof}

\begin{lemma}
\label{lemma-horizontal}
There exist principal connections $A_L$ and $A_R$ on the bundles $\mathcal{P}_L$ and $\mathcal{P}_R$ (respectively) such that for corresponding horizontal curves we have $\gamma_L(t) = \gamma_R(t)$.
\end{lemma}

\begin{proof}
Let $\theta$ be the canonical Liouville one-form on $T^*G$.
Consider a curve $c_L : [0, t] \rightarrow \mathcal{P}_L$, such that $c_L(0) = p_0$
and $\rho_L \circ c_L = p$. Let $f = \pi(c_L(t))$.

Now we construct the curve $c_R : [0, t] \rightarrow \mathcal{P}_R$ such that $\rho_R \circ c_R = p^-$ and with the following properties:
\begin{equation}
\label{eq-cR-conditions}
c_R(0) = p_t, \qquad c_R(t) = c_L(t) \in \Gamma, \qquad \theta(\dot{c}_R(\tau)) = \theta(\dot{c}_L(t-\tau)), \qquad \tau \in [0, t].
\end{equation}

First, consider a curve $C_R : [0, t] \rightarrow \mathcal{P}_R$ such that
\begin{equation}
\label{eq-cR-boundary-conditions}
C_R(0) = p_t, \quad C_R(t) = c_L(t), \quad
d\pi\dot{C}_R(0) = L_{f^{-1}} d\pi \dot{c}_L(t), \quad d\pi \dot{C}_R(t) = R_f d\pi \dot{c}_L(0).
\end{equation}

Second, consider $\Delta_{\tau} = \theta(\dot{C}_R(\tau)) - \theta(\dot{c}_L(t-\tau)) \in \R$. It follows from~\eqref{eq-cR-boundary-conditions} that $\Delta_0 = \Delta_t = 0$.
Define the family of elements $d_{\tau} = \exp{\bigl(\frac{1}{p_t(\xi_R)} \Delta_{\tau} \xi_R\bigr)} \in G_{p_t}$,
where $\sspan{\xi_R} = \g_{p_t}$.
Notice, that
\begin{equation}
\label{eq-connections-factor-coinscide}
p_t(\xi_R) = (\Ad^*{f^{-1}}) p_0 ((\Ad{f^{-1}})\xi_L) = p_0(\xi_L) \neq 0.
\end{equation}
Consider the curve $c_R(\tau) = R(d_{\tau}) C_R(\tau)$.
We obtain
$$
\theta(\dot{c}_R(\tau)) = \theta(R(d_{\tau})\dot{C}_R(\tau)) - \frac{1}{p_t(\xi_R)} \Delta_{\tau} \theta(L_{\pi(C_R(\tau))} \xi_R).
$$
Since $\theta(R(d_{\tau})\dot{C}_R(\tau)) = \theta(\dot{C}_R(\tau))$ and
$\theta(L_{\pi(C_R(\tau))} \xi_R) = (\Ad^*{\pi(C_R(\tau))^{-1}}) p_{t-\tau} (\xi_R) = p_t(\xi_R)$, we get
$\theta(\dot{c}_R(\tau)) = \theta(\dot{C}_R(\tau)) - \Delta_{\tau} = \theta(\dot{c}_L(t-\tau))$.
So, since $d_0 = d_t = \id$, the curve $c_R$ satisfies the conditions~\eqref{eq-cR-conditions}.

Let $\g_{p_0} = \sspan{\xi_L}$, where $\xi_L = (\Ad{f})\xi_R \in \g_{p_0}$.
Define the principal connections on the bundles
$\mathcal{P}_L$, $\mathcal{P}_R$ as
\begin{equation}
\label{eq-connections}
A_L = \frac{1}{p_0(\xi_L)} \theta \otimes \xi_L, \qquad A_R = -\frac{1}{p_t(\xi_R)} \theta \otimes \xi_R.
\end{equation}
(We have the sign minus in definition of $A_R$ because the projection $d\pi$ of the velocity field of the right $G_{p_t}$-action is a left-shift of the vector $-\xi_R$.)
Notice, that the factors of these connections coincide one with another~\eqref{eq-connections-factor-coinscide}.

Let $q_L \in G_{p_0}$ be such that $\gamma_L(t) = L(q_L) c_L(t)$. Let $q_R \in G_{p_t}$ be such that $\gamma_R(t) = R(q_R) c_R(t)$. We have
$$
q_L = \exp{\biggl(\int\limits_0^t {A_L(\dot{c}_L(\tau))\, d\tau}\biggr)}, \qquad q_R = \exp{\biggl(\int\limits_0^t{A_R(\dot{c}_R(\tau))\, d\tau}\biggr)}.
$$
Now from equations~\eqref{eq-cR-conditions}, \eqref{eq-connections-factor-coinscide}, \eqref{eq-connections} it follows that
$q_R = f^{-1}(q_L^{-1})f$.

We obtain
$$
\pi(\gamma_R(t)) = R_{q_R^{-1}} \pi(c_R(t)) = f q_R^{-1} = f f^{-1} q_L f = L_{q_L} \pi(c_L(t)) = \pi(\gamma_L(t)).
$$
But $\gamma_L(t), \gamma_R(t) \in \Gamma$, thus $J_R(\gamma_L(t)) = J_R(\gamma_R(t)) = p_t$, and $\gamma_L(t)$ and $\gamma_R(t)$ are uniquely defined by their projections to the Lie group $G$. So, $\gamma_L(t) = \gamma_R(t)$.
\end{proof}

\begin{lemma}
\label{lemma-addition}
Consider the curves $h_L$ and $h_R$ defined by~\eqref{eq-addition-L}--\eqref{eq-addition-R}, then $h_R(t) = r^{-1}h_L(t)^{-1}r$,
where $r = \pi(\gamma_L(t))$.
\end{lemma}

\begin{proof}
The differential equations for $h_L$ and $h_R$ read as follows:
\begin{align*}
&\dot{h}_L(\tau) = h_L(\tau) \eta_L(\tau), \qquad h_L(0) = \id, \qquad \ \eta_L(\tau) = A_L(\Hl(\lambda_L(\tau))), \\
&\dot{h}_R(\tau) = h_R(\tau) \eta_R(\tau), \qquad h_R(0) = \id, \qquad \eta_R(\tau) = A_R(\Hr(\lambda_R(\tau))).
\end{align*}
By definitions of connections $A_L$, $A_R$ and Hamiltonian vector fields $\Hl$, $\Hr$ we obtain
\begin{align*}
&\eta_L(\tau) = \frac{\xi_L}{p_0(\xi_L)} (L_{g_L(\tau)^{-1}}^* p_{\tau}) (L_{g_L(\tau)} d_{p_{\tau}} H) = \frac{\xi_L}{p_0(\xi_L)} p_{\tau}(d_{p_{\tau}} H), \\
&\eta_R(\tau) = -\frac{\xi_R}{p_t(\xi_R)} (R_{g_R(\tau)^{-1}}^* p_{t-\tau}) (R_{g_R(\tau)} d_{p_{t-\tau}} H) = -\frac{\xi_R}{p_t(\xi_R)} p_{t-\tau}(d_{p_{t-\tau}} H),
\end{align*}
where $g_L(\tau) = \pi(\lambda_L(\tau))$ and $g_R(\tau) = \pi(\lambda_R(\tau))$.

Since $G_{p_0}$ and $G_{p_t}$ are commutative groups, using~\eqref{eq-connections-factor-coinscide} we have
$$
h_L(t) = \exp\biggl(\int\limits_0^t{\eta_L(\tau)\, d\tau}\biggr) = \exp(I \xi_L), \qquad
h_R(t) = \exp\biggl(\int\limits_0^t{\eta_R(\tau)\, d\tau}\biggr) = \exp(-I \xi_R),
$$
where $I = \frac{1}{p_0(\xi_L)}\int_0^t{p_{\tau}(d_{p_{\tau}}H)\, d\tau}$. It follows that
$$
h_R(t) = f^{-1} h_L^{-1}(t) f = r^{-1} (rf^{-1}) h_L^{-1}(t) (rf^{-1})^{-1} r.
$$
Since $rf^{-1} \in G_{p_0}$, it commutate with $h_L^{-1}(t)$, we get $h_R(t) = r^{-1} h_L^{-1}(t) r$.
\end{proof}

To finish the proof of Proposition~\ref{prop-flows-of-left-and-right-Hamiltonian-vector-fields} we obtain
$$
\pi(\lambda_R(t)) = \pi(\gamma_R(t)) h_R(t)^{-1} = (\pi(\gamma_R(t)) r^{-1})h_L(t)r.
$$
Notice that $(\pi(\gamma_R(t)) r^{-1}), h_L(t) \in G_{p_0}$, so these two elements commutate. Thus
$$
\pi(\lambda_R(t)) = h_L(t) (\pi(\gamma_R(t)) r^{-1})r = h_L(t) \pi(\gamma_L(t)) = \pi(\lambda_L(t)).
$$
Since $\lambda_L(t), \lambda_R(t) \in \Gamma$, these elements are uniquely defined by their projections to the Lie group $G$. So, $\lambda_L(t) = \lambda_R(t)$, and we finish the proof of Proposition~\ref{prop-flows-of-left-and-right-Hamiltonian-vector-fields}.

Finally, notice that $\{ p \in \g^* \ | \ p(\g_p) \neq 0 \}$ is an open dense subset in $\g^*$.
Consider the open dense subset $\mathcal{W} = \{ p \in \g^* \ | \ \codim{(\Ad^*{G})p} = 1, \ p(\g_p) \neq 0 \} \subset \g^*$. Define symmetry $s$ on this subset in the pre-image of exponential map.
By Proposition~\ref{prop-flows-of-left-and-right-Hamiltonian-vector-fields} we have
$e^{t\Hl} (p) = e^{t\Hr} \circ e^{t\Hv} (p)$ for $p \in \mathcal{W}$. From Lemmas~\ref{lemma-new-flow}, \ref{lemma-case-b-changes-Hamiltonian-to-right-invariant} it follows that this is equal to
$(S^*)^{-1} \circ e^{t\Hl} \circ S^* \circ e^{t\Hv}$, and we obtain
$$
e^{t\Hl} \circ S^* \circ e^{t\Hv} (p) = S^* \circ e^{t\Hl} (p) \quad \text{for} \quad p \in \mathcal{W}.
$$
Now Theorem~\ref{th-symmetry} follows from~\eqref{eq-projection}. $\Box$

\section{\label{section-example}Action of symmetries on a semi-direct product}

Applications of the geometric control theory often deal with left-invariant control problems on semi-direct products of Lie groups.
The considered method was successfully applied to sub-Riemannian problems on the Lie groups $\SE_2 = \R^2\leftthreetimes\SO_2$ (Yu.\,L.~Sachkov, I.~Moiseev~\cite{sachkov-moiseev-se-1, sachkov-se-2, sachkov-se-3}), $\SH_2 = \R^2\leftthreetimes\R$ (Ya.\,A.~Butt, Yu.\,L.~Sachkov, A.\,I.~Bhatti~\cite{sh2-1, sh2-2}), $\SO_3\times\R^2$ (the problem of a rolling sphere on the plane without twisting and slipping, Yu.\,L.~Sachkov~\cite{sachkov-sphere-plane}). In these papers an action of symmetries of exponential map on endpoints of extremal trajectories was described. From Theorem~\ref{th-symmetry} we get a general formula for this action.

\begin{corollary}
\label{crl-semidirect}
Let $G = G_1 \leftthreetimes_b G_2$, where $b : G_2 \rightarrow \Aut{G_1}$ is a homomorphism. Under conditions of Theorem~$\ref{th-symmetry}$ a symmetry of the exponential map acts on the Lie group $G$ in the following way:
$$
S^{-1} (g_1, g_2) =
\left\{
\begin{array}{rl}
(S^{-1}g_1, S^{-1}g_2), & \text{in case~\emph{(a)}}, \\
(b(S^{-1}g_2)(S^{-1}g_1), S^{-1}g_2), & \text{in case~\emph{(b)}}, \\
\end{array}
\right.
$$
\end{corollary}

\begin{proof}
Immediately follows from Theorem~\ref{th-symmetry} and multiplication law in the group $G$, that reads as
$(g_1, g_2) \cdot (h_1, h_2) = (g_1 b(g_2) h_1, g_2 h_2)$.
\end{proof}

\begin{example}
Corollary~\ref{crl-semidirect} allows us to clarify the geometric sense of the Maxwell strata that were found in paper~\cite{sachkov-moiseev-se-1} for the sub-Riemannian problem on the Lie group $\SE_2$ (group of proper isometries of Euclidean plane). In other words this sub-Riemannian problem models an optimal parking of unicycle.

A point $(v, R_{\varphi}) \in \R^2 \leftthreetimes \SO_2$ is a fixed point of a symmetry of case~(b) iff $S(R_{\varphi}) = R_{\varphi}$ and $S(v) = R_{-\varphi}v$, where $R_{\varphi}$ is rotation by angle $\varphi$. We have two possibilities.

The first one is when $S|_{\SO_2}$ is the inversion. In this case if $S(R_{\varphi}) = R_{\varphi}$, then $\varphi = 0$ or
$\varphi = \pi$. So, $S|_{\R^2}$ equals $\id$ or $-\id$, respectively. We have two possible sets of fixed points: the set of translations and the set of central symmetries. Both these sets are disks that appear as Maxwell strata.

The second possibility is $S|_{\SO_2} = \id$. Then $S|_{\R^2}$ is a reflection with respect to a line $l_S$.
From $S(v) = R_{-\varphi}v$ it follows that the set of fixed points is the set of rotations around points located on the line $l_S^\bot$ that is orthogonal to the line $l_S$ (including translations along the line $l_S$, considered as rotations with infinite center at the line $l_S^\bot$). This set is a M\"obius strip (see Lemma~\ref{lemma-Mobius-strip} below).

But only two M\"obius strips appear as Maxwell strata. Because our symmetry is not just anti-automorphism, but it is induced from a symmetry of the vertical part of the Hamiltonian vector field. These M\"obius strips correspond to two orthogonal lines $x$ and $y$ in $\R^2$, where the base vector of the line $x$ indicates the initial position of unicycle.
\end{example}

\begin{lemma}
\label{lemma-Mobius-strip}
The set of rotations around centers that lies on a fixed line (including infinite center) is a M\"obius strip.
\end{lemma}

\begin{proof}
Consider a cylinder $[-\infty, +\infty] \times S^1$ of pairs consisting of a center of rotation and an angle of rotation. To construct the set mentioned above from this cylinder one needs to perform the following steps (see Figure~\ref{pic-moebius}):\\
1. identify the bases of the cylinder (by identifying opposite points on the base circles);\\
2. identify all points of the circle $S^1_0$, corresponding to rotations by zero angle, because all points of this circle generate the same transformation of the plane;\\
3. remove the point $T_{\infty}$ (on the identified bases) that corresponds to the infinite translation.

\begin{figure}[h!]
\caption{The M\"obius strip consisting of rotations around centers that are on a fixed line. Base circles are identified by the directions of arrows.}
\label{pic-moebius}
\center{\includegraphics[height=0.2\linewidth,width=0.75\linewidth]{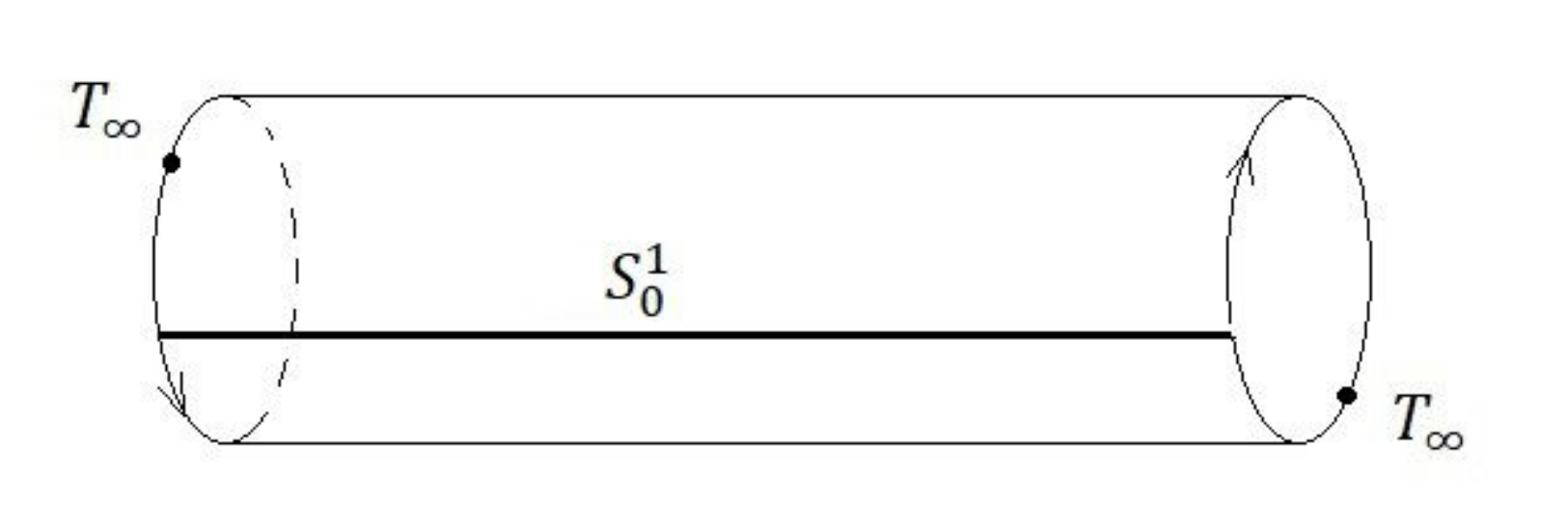}}
\end{figure}

We get a Klein bottle after step~1. To make step~2, cut the bottle along a circle, we get a M\"obius strip. Identify all its boundary points, i.e., glue this M\"obius strip and a disk along their boundaries. We get a projective plane. It remains to cut one point (step~3), we get a M\"obius strip.
\end{proof}

\begin{example}
The geometric sense of the Maxwell strata for sub-Riemannian problem on the Lie group $\SH_2$~\cite{sh2-1} is quite similar to $\SE_2$-case.
There are the set of translations and two sets of hyperbolic rotations around points located on two orthogonal lines
(including infinite centers of rotations).
All of these sets are homeomorphic to $\R^2$.
\end{example}

\end{document}